\documentclass{article}

\usepackage{graphicx}
\usepackage{subcaption}
\usepackage{xcolor}
\usepackage{amsthm}
\usepackage{amsmath}
\usepackage{amssymb}
\usepackage{fullpage}
\usepackage{hyperref}

\newtheorem{definition}{Definition}[section]
\newtheorem{lemma}[definition]{Lemma}

\newtheorem{theorem}[definition]{Theorem}
\newtheorem{remark}[definition]{Remark}
\newtheorem{question}[definition]{Question}

\DeclareMathOperator{\sign}{sign}

\graphicspath{{figures/}}

\title{A counterexample to Payne's nodal line conjecture with few holes}
\author{Joel Dahne \and Javier G\'omez-Serrano \and Kimberly Hou}

\begin{document}
\maketitle
\begin{abstract}
  Payne conjectured in 1967 that the nodal line of the second
  Dirichlet eigenfunction must touch the boundary of the domain. In
  their 1997 breakthrough paper, Hoffmann-Ostenhof, Hoffmann-Ostenhof
  and Nadirashvili proved this to be false by constructing a
  counterexample in the plane with many holes and raised the question
  of the minimum number of holes a counterexample can have. In this
  paper we prove it is at most 6.
\end{abstract}

\section{Introduction}
\label{sec:introduction}

Let $\Omega$ be a bounded planar domain, and let $\lambda_i$ be the eigenvalues of the Dirichlet Laplacian, namely the numbers  $0 < \lambda_1 < \lambda_2 \leq \lambda_3 \leq \ldots$ that satisfy
\begin{align}
 -\Delta u_k & = \lambda_k u_k  \text{ in } \Omega \nonumber \\
u_k & = 0  \text{ on } \partial \Omega. \label{laplacian-pde}
\end{align}

The so-called ``nodal line conjecture'' from 1967 by Payne \cite[Conjecture 5, p.467]{Payne:isoperimetric-inequalities}, \cite{Payne:two-conjectures-fixed-membrane} says that the nodal line (the zero level set) of $u_2$ on a bounded domain in $\mathbb{R}^{2}$ must touch the boundary. This statement was later extended by Yau \cite{Yau:open-problems} to the higher dimensional case. Contrary to its apparent simplicity, the conjecture is still open in the general case, though a few results (both positive and negative) and extensions have been done. If the domain is convex, Melas \cite{Melas:nodal-line-conjecture-convex} (in the case of $C^\infty$ boundary) and Alessandrini \cite{Alessandrini:nodal-lines-eigenfunctions-convex} (in the general case) proved it in the positive. Jerison \cite{Jerison:nodal-line-convex-planar-domain} proved the conjecture for long thin convex sets, and Jerison \cite{Jerison:diameter-nodal-line-convex-domain}, Grieser--Jerison \cite{Grieser-Jerison:asymptotics-nodal-line-convex} and Beck--Canzani--Marzuola \cite{Beck-Canzani-Marzuola:nodal-line-estimates} gave more information on the location of the nodal line. Under various additional symmetry and/or convexity assumptions, the conjecture has been proved by Payne \cite{Payne:two-conjectures-fixed-membrane}, Lin \cite{Lin:second-eigenfunction-laplacian}, P\"utter \cite{Putter:nodal-lines-second-eigenfunctions}, Damascelli \cite{Damascelli:nodal-set-second-eigenfunction} and Yang--Guo \cite{Yang-Guo:payne-conjecture-concave-domains}. See also \cite{Grebenkov-Nguyen:structure-eigenfunctions-survey}.

In 1997, Hoffmann-Ostenhof--Hoffmann-Ostenhof--Nadirashvili \cite{HoffmannOstenhof-HoffmannOstenhof-Nadirashvili:nodal-line-payne-counterexample} constructed a counterexample of a planar, bounded, non-simply connected domain for which the nodal line is closed and does not touch the boundary. Their construction was extended by Fournais \cite{Fournais:nodal-surface-closed-Rd} and later by Kennedy \cite{Kennedy:closed-nodal-surfaces-higher-dimensions} for the higher dimensional case. Freitas--Krej\v{c}i\v{r}\'{\i}k \cite{Freitas-Krejcirik:unbounded-domains-payne-conjecture} constructed an unbounded counterexample. Kennedy \cite{Kennedy:toy-neumann-nodal-line} proved similar results (both positive and negative) for a toy Neumann analogue under the assumption of central symmetry of the domain. In both \cite{HoffmannOstenhof-HoffmannOstenhof-Nadirashvili:nodal-line-payne-counterexample,Fournais:nodal-surface-closed-Rd} the main domain construction is based on ``carving'' a sufficiently large number $N$ of holes on a symmetric domain so that the nodal set gets disconnected from the boundary, but no quantitative estimates are provided on $N$ (Kennedy's construction \cite{Kennedy:closed-nodal-surfaces-higher-dimensions} of a simply connected domain is valid in dimensions 3 or higher). Specifically, in \cite{HoffmannOstenhof-HoffmannOstenhof-Nadirashvili:nodal-line-payne-counterexample} the number of holes is claimed to be \textit{delicate to bound} and in \cite{Fournais:nodal-surface-closed-Rd} \textit{of the order of $10^9$}. Hoffmann-Ostenhof, Hoffmann-Ostenhof and Nadirashvili raised the following question circa 25 years ago:

\begin{question}[\cite{HoffmannOstenhof-HoffmannOstenhof-Nadirashvili:nodal-line-payne-counterexample}, Remark 3]
\label{remhohon}
Clearly, an interesting question is whether there exists a simply connected domain for which the second eigenfuntion has a closed nodal line. We do not believe this. So a more general question is: What is the smallest $N_0$ such that there exists a domain with $N_0$ boundary components whose second eigenfunction has a nodal line that does not hit the boundary?
\end{question}

In this paper we provide a partial answer to Question \ref{remhohon},
giving an upper bound of 7 by constructing an example of a domain with
a closed nodal line (see Figure \ref{fig:domain} for an illustration
of the domain and Figure \ref{fig:nodal-line} for the nodal line). Our
main result is the following theorem:

\begin{theorem}\label{MainThm}
There exists a planar domain with 6 holes for which the nodal line of $u_2$ is closed.
\end{theorem}

\begin{remark}
We make no claim that $N_0 = 7$ is optimal. Possibly with a more thorough search, more computational time and better estimates one can get to $N_0 = 5$ or lower.
\end{remark}

We summarize the main steps of the proof. The underlying idea is that if we can find a suitable candidate domain and we can find an approximate $u_2$ (in the sense that it satisfies equation \eqref{laplacian-pde} up to a small error) which has a closed nodal line, then there has to exist an exact solution nearby also having a closed nodal line, and no other spurious nodal domains can appear since $u_2$ can have at most two nodal domains. The difficult part is to derive effective (as opposed to asymptotic, either with respect to the number of holes or without explicit constants) stability estimates that are good enough, as well as to find good candidates since the problem seems to be quite unstable in that regard. Moreover, since the candidate domains are far from any explicit domains for which the spectrum and the nodal lines are known, perturbative methods fail. Furthermore, the errors need to be very small due to the nodal line being very close to the boundary. On top of that, the eigenvalue $\lambda_2$ is very close to $\lambda_3$ and $\lambda_4$ which makes difficult to distinguish the eigenfunction associated with the former from the latter two. To overcome these difficulties parts of the proof will be computer-assisted. Nevertheless, one has to proceed very carefully and derive extremely tight bounds to make all the estimates work.


\subsection{The Method of Particular solutions and Validation of Eigenvalues}

Perturbative analysis of the eigenpairs is a classical problem, and
there is an extensive
literature~\cite{Kato:upper-lower-bounds-eigenvalues,Lehmann:optimale-eigenwerte,Behnke-Goerisch:inclusions-eigenvalues,Fox-Henrici-Moler:approximations-bounds-eigenvalues,Moler-Payne:bounds-eigenvalues,Still:computable-bounds-eigenvalues,Barnett-Hassell:quasi-ortogonality-dirichlet-eigenvalues}
showing the existence of an eigenvalue close to an approximate one.
The theme of these results is that if one can find
$(\lambda_\text{app},u_\text{app})$ satisfying the equation up to an
error bounded by $\delta$, then one can show that there is a true
eigenpair $(\lambda,u)$ at a distance $C\delta^k$. The main drawback
is that these methods can not tell the position of the eigenvalue within
the spectrum.

To solve this issue, Plum \cite{Plum:eigenvalues-homotopy-method} proposed a homotopy method linking the eigenvalues of the domain of the problem with another, known domain (the base problem). See also the intermediate method \cite{Weinstein-Stenger:eigenvalues-book,Goerisch:stufenverfahren-eigenwerten,Beattie-Goerisch:lower-bounds-eigenvalues} for another example of connecting the problem to a known domain. Our domain is very far from any known domain (due to the holes), so we have opted for a more direct approach. Using domain monotonicity with inclusion also did not yield good enough (lower) bounds for our purpose. Using the Finite Element framework, Liu and Liu--Oishi \cite{Liu-Oishi:verified-eigenvalues-laplacian-polygons}, \cite{Liu:framework-verified-eigenvalues} have managed to give explicit, rigorous computable lower bounds of the spectrum in terms of solutions of a (big) finite linear system (see also \cite{Carstensen-Gedicke:lower-bounds-eigenvalues} for similar bounds, \cite{You-Xie-Liu:guaranteed-bounds-steklov} for the case of the Steklov problem and \cite{Liu:eigenvalue-bounds-differential-operators} for a more general setting). Nonetheless, the aforementioned bounds are not good enough to obtain Theorem \ref{MainThm} as the eigenvalues are tightly clustered and that would require a mesh so refined we would not be able to handle it on a computer within a reasonable time.

In this paper we will combine the two families as in
\cite{GomezSerrano-Orriols:negative-hearing-shape-triangle}. The first
pass will separate the first 4 eigenvalues from the rest of the
spectrum (using the method of
\cite{Liu:framework-verified-eigenvalues}). The enclosures and the
scale are coarse at this point. The second pass will find 4
approximate eigenpairs below the threshold and use the finer stability
methods to separate \(\lambda_{2}\) from the others. There is a big
technical difficulty since $\lambda_3$ and $\lambda_4$ presumably
correspond to a double eigenvalue, and we have to handle the
circumstance that the span is two-dimensional in that case.

In order to find accurate approximations of the eigenvalues and
eigenfunctions we will use the Method of Particular Solutions (MPS).
This method was introduced by Fox, Henrici and Moler
\cite{Fox-Henrici-Moler:approximations-bounds-eigenvalues} and has
been later adapted by many authors (see
\cite{Antunes-Valtchev:mfs-corners-cracks,Read-Sneddon-Bode:series-method-mps,Golberg-Chen:mfs-survey,Fairweather-Karageorghis:mfs-survey,Betcke:generalized-svd-mps}
as a sample, and the thorough review
\cite{Betcke-Trefethen:method-particular-solutions}). The main idea is
to consider a set of functions that solve the eigenvalue problem
without boundary conditions as a basis, and writing the solution of
the problem with boundary conditions as a linear combination of them,
solving for the coefficients that minimize the error on the boundary.
Typically, the choices have been rational functions
\cite{Hochman-Leviatan-White:rational-function-laplacian} or products
of Bessel functions and trigonometric polynomials centered at certain
points. The different choices of these functions have a big impact on
the performance of the method. Recently, Gopal and Trefethen
\cite{Gopal-Trefethen:new-laplace-solver-pnas,Gopal-Trefethen:laplace-solver-detailed}
have developed a new way of selecting the base functions in such a way
as to yield root exponential convergence (the \textit{lightning
  Laplace solver}). We stress that these methods produce accurate
approximations but there is no explicit control of the error with
respect to the true solution. This is handled a posteriori with a
perturbative analysis of the approximations.

The use of computers to prove mathematically rigorous theorems has
become increasingly popular in the last 20 years and many goals and
theories have been developed in this blooming field. Floating point
arithmetic errors are handled and controlled via interval arithmetic,
where real numbers are replaced by real intervals and all the errors
are propagated throughout the calculations. We refer to the book
\cite{Tucker:validated-numerics-book} for an introduction to validated
numerics, and to the survey \cite{GomezSerrano:survey-cap-in-pde} and the recent
book \cite{Nakao-Plum-Watanabe:cap-for-pde-book} for a more specific
treatment of computer-assisted proofs in PDE. We also mention the work
of Tanaka \cite{Tanaka:aposteriori-sign-change-elliptic-pde} where he
also controls the nodal line of an elliptic problem (as opposed to an
eigenvalue problem in our case) using computer-assisted techniques,
and Dahne--Salvy \cite{Dahne-Salvy:enclosures-eigenvalues} enclosing
eigenvalues of the Laplacian on spherical triangles with techniques
very similar to those used here.

The paper is organized as follows. In Section
\ref{sec:finding-candidate} we present the candidate for the counterexample and discuss how it was found. In Section
\ref{sec:separating-first-four} we explain the separation of the first
four eigenvalues from the rest. Section
\ref{sec:constructing-approximation} gives details of how the
approximate eigenfunctions are constructed and in Section
\ref{sec:isolating-second} we make use of these approximations for
isolating the second eigenvalue. Finally in Section
\ref{sec:nodal-line} we prove that the nodal line of the second
eigenfunction is closed and conclude the proof of Theorem \ref{MainThm}. Section \ref{sec:details-of-implementation}
contains details about the implementation of the computer assisted
parts.

\section{Finding a candidate}
\label{sec:finding-candidate}

\begin{figure}
  \centering
  \begin{subfigure}[t]{0.45\textwidth}
    \includegraphics[width=\textwidth]{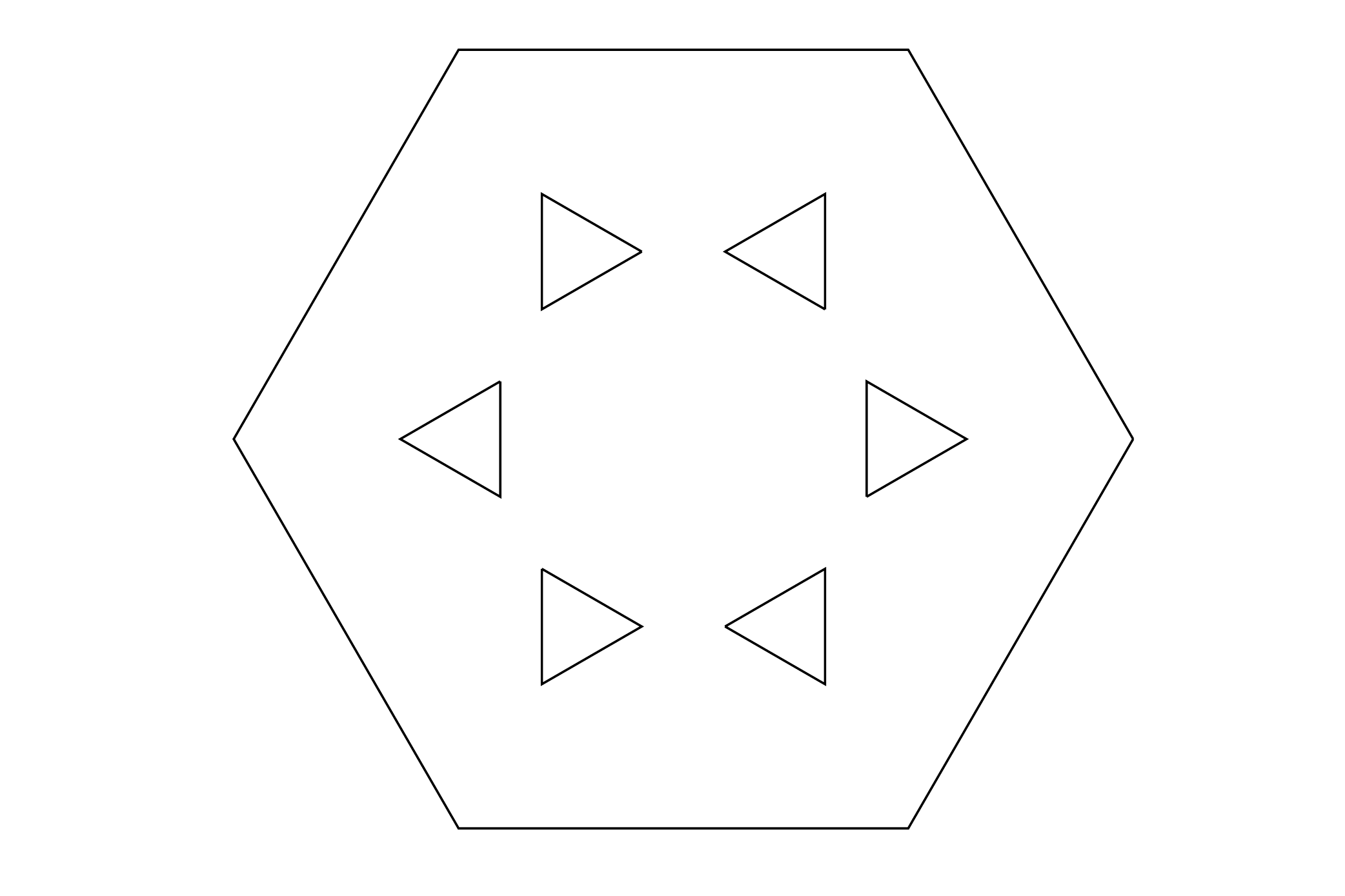}
    \caption{The candidate domain is a hexagon with side length \(1\)
      with holes that are equilateral triangles of height
      \(\frac{6}{27}\) with centers placed at a distance of
      \(\frac{11}{27}\) from the origin.}
    \label{fig:domain}
  \end{subfigure}
  \hspace{0.05\textwidth}
  \begin{subfigure}[t]{0.45\textwidth}
    \includegraphics[width=\textwidth]{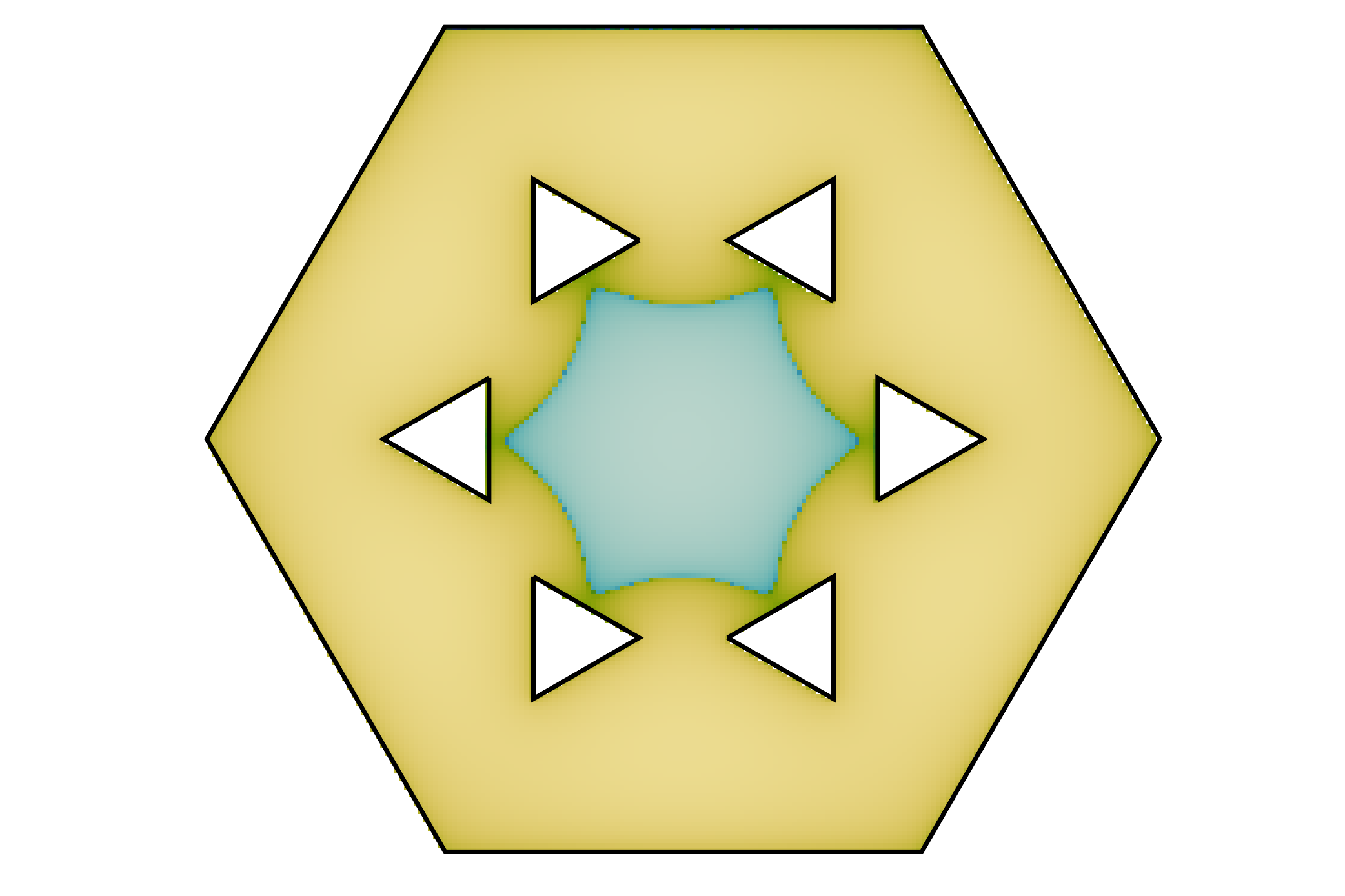}
    \caption{Plot highlighting the nodal line of \(\tilde{u}_{2}\),
      generated by plotting
      \(\sign(\tilde{u}_{2})\log(|\tilde{u}_{2}|)\).}
    \label{fig:nodal-line}
  \end{subfigure}
  \caption{The counterexample and an approximation of its nodal
    domains.}
\end{figure}

The main idea behind the choice of the domain was to start from
Hoffmann-Ostenhof--Hoffman-Ostenhof--Nadirashvili's construction from
a disk and carve as few (but possibly large) holes as possible. We
tried to work with domains as symmetric as possible and holes with few
sides to reduce the computational cost. However, due to the lower bound of Theorem \ref{thm:1} being restricted to polygonal
domains we chose the domain to be of polygonal shape. We found many
instances of domains for which the nodal line was closed, though the
problem seems to be quite sensitive to the position and shape of the
holes, and small perturbations destroy the closedness of the nodal
line due to the very small nature of the relevant numbers, see Figure
\ref{fig:candidates} for a few different candidates. In the end we
settled for a domain given by a hexagon with six holes in it. The side
length of the hexagon is normalized to 1 and the holes are equilateral
triangles of height \(\frac{6}{27}\) with their centers placed a
distance \(\frac{14}{27}\) from the origin, see Figure
\ref{fig:domain}. The search was done using the PDE Toolbox in Matlab
and the first five approximate eigenvalues for the final domain were
\begin{equation}
  \label{eq:7}
  \lambda_{1} = 31.0432,\
  \lambda_{2} = 63.2104,\
  \lambda_{3} = 63.7259,\
  \lambda_{4} = 63.7259,\
  \lambda_{5} = 68.2629.
\end{equation}

\begin{figure}
  \centering
  \includegraphics[width=\textwidth]{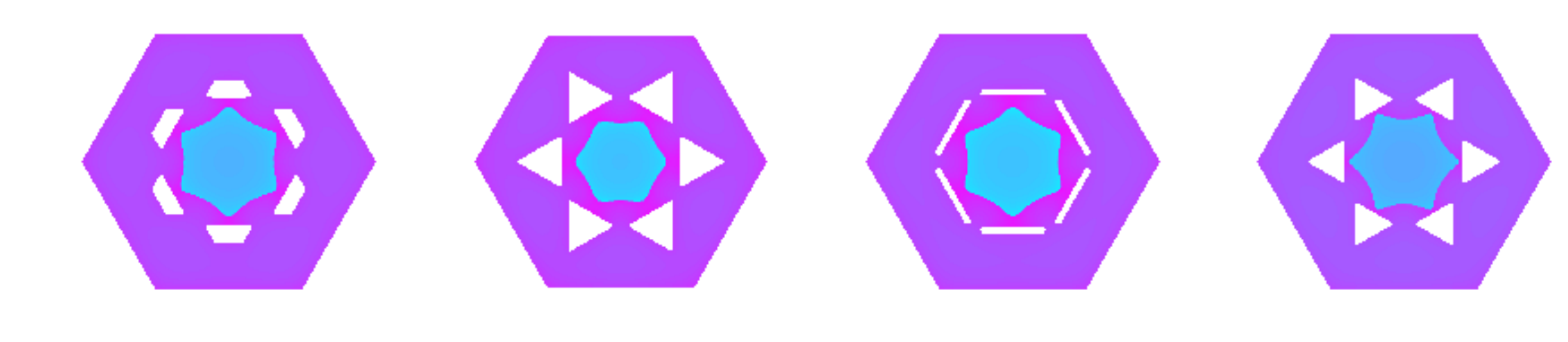}
  \caption{Four domains for which the nodal line seems to be closed.
    Similarly to Figure \ref{fig:nodal-line} we plot
    \(\sign(u)\log(|u|)\). The rightmost one is the one we finally
    chose and the only one for which we have proved that the nodal
    line indeed is closed.}
  \label{fig:candidates}
\end{figure}

\section{Separating the first four eigenvalues}
\label{sec:separating-first-four}
The first four eigenvalues will be separated from the rest by
computing a lower bound for the fifth eigenvalue. We make use of
recent results by \cite{Liu:framework-verified-eigenvalues}. The
procedure is exactly the same as in
\cite{GomezSerrano-Orriols:negative-hearing-shape-triangle}.

The starting point is a triangulation of the domain as given in Figure
\ref{fig:mesh}. It consists mostly of equilateral triangles except
next to the boundary of the holes where the triangles are cut in half
(these are colored red in the figure). The basis functions are indexed
by the interior edges of the triangulation: For an edge \(E\) between
two triangles \(\tau_{1}\) and \(\tau_{2}\) the basis function
\(\psi_{E}\) is the unique function supported on
\(\tau_{1} \cup \tau_{2}\) such that the restriction to each triangle
is affine, takes the value 1 at the midpoint of \(E\) and the value 0
at the midpoints of the other edges of \(\tau_{1}\) and \(\tau_{2}\).
In our case there are three kinds of edges, type 1 where \(\tau_{1}\)
and \(\tau_{2}\) are both equilateral (these correspond to the
majority of edges), type 2 where only one of them is equilateral and
type 3 where none of them are equilateral.

\begin{figure}
  \centering
  \begin{subfigure}[t]{0.45\textwidth}
    \includegraphics[width=\textwidth]{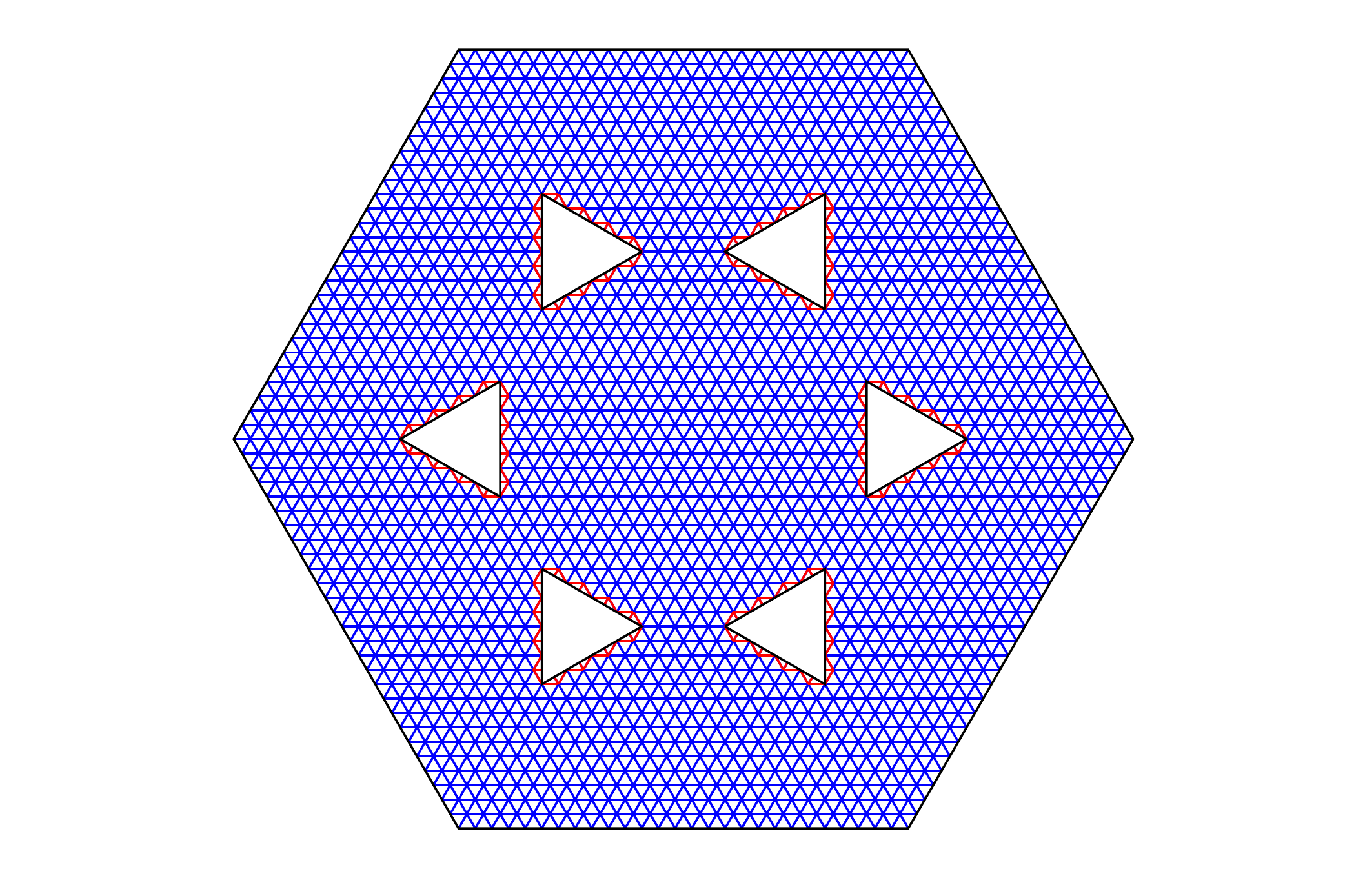}
    \caption{Triangulation of the domain.}
    \label{fig:mesh}
  \end{subfigure}
  \hspace{0.05\textwidth}
  \begin{subfigure}[t]{0.45\textwidth}
    \includegraphics[width=\textwidth]{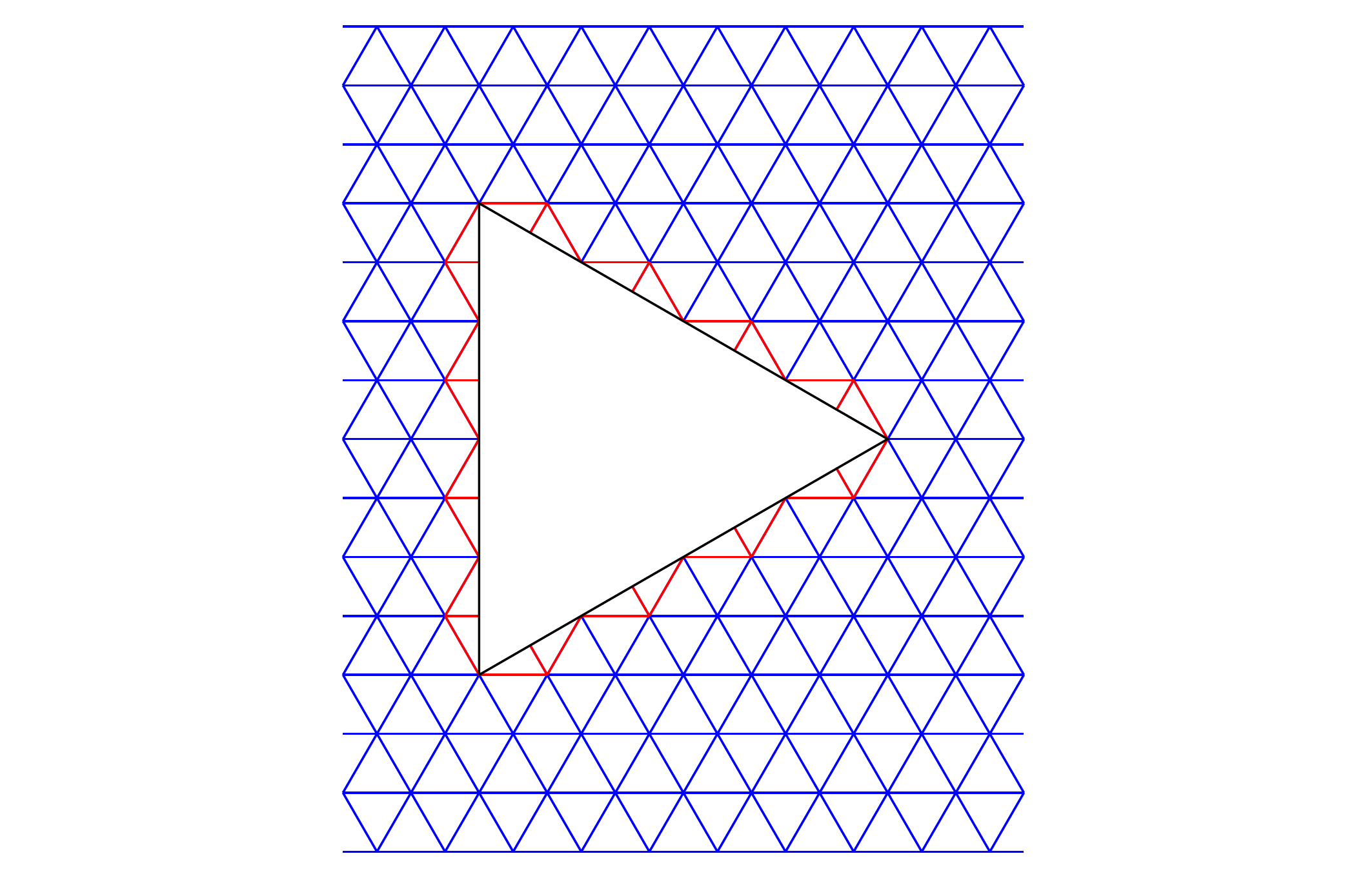}
    \caption{Triangulation of the domain around the holes.}
    \label{fig:mesh-zoomed}
  \end{subfigure}
  \caption{Meshing of the counterexample.}
\end{figure}

The weak formulation of the problem \eqref{laplacian-pde} reads
\begin{equation*}
  \int_{\Omega}\nabla u \cdot \nabla \psi = \lambda \int_{\Omega} u \psi.
\end{equation*}
We define the coefficients of the stiffness and mass matrices
\(A = (a_{EF}), B = (b_{EF})\) by
\begin{equation*}
  a_{EF} = \int_{\Omega}\nabla \psi_{E} \cdot \nabla \psi_{F}, \quad
  b_{EF} = \int_{\Omega}\psi_{E}\psi_{F}
\end{equation*}
which leave us with solving discrete system
\begin{equation}
  \label{eq:1}
  Ax = \lambda Bx.
\end{equation}
Here the vector \(x = (x_{E})\) corresponds to the discrete solution
\(u = \sum x_{E}\psi_{E}\). For our choice of triangulation the mass
matrix will be diagonal with coefficients
\begin{equation*}
  \frac{h^{2}}{2\sqrt{3}}, \quad
  \frac{3h^{2}}{8\sqrt{3}}, \quad
  \frac{h^{2}}{4\sqrt{3}}, \quad
\end{equation*}
for edges of type 1, 2 and 3 respectively. Here \(h\) is the side
length of the equilateral triangles. The matrix \(A\) is sparse and on
the diagonal the coefficients are
\begin{equation*}
  \frac{8}{\sqrt{3}}, \quad
  \frac{12}{\sqrt{3}}, \quad
  \frac{4}{\sqrt{3}}, \quad
\end{equation*}
for edges of type 1, 2 and 3 respectively. The off-diagonal entries
are zero if the two edges do not have a common triangle or
\(\frac{-2}{\sqrt{3}}\) if they do. Since \(B\) is diagonal it is
easily inverted, which allows us to reduce the problem to solving the
matrix eigenvalue problem
\begin{equation}
  \label{eq:2}
  Mx = \lambda x,\quad M = B^{-1}A.
\end{equation}

The solution of the finite element problem can then be linked to the
continuous problem using the following Theorem:
\begin{theorem}\cite[Theorem 2.1, Remark 2.2]{Liu:framework-verified-eigenvalues}
  \label{thm:1}
  Consider a polygonal domain \(\Omega\) with a triangulation so that
  each triangle has diameter at most \(h\). Let \(\lambda_{k}\) be the
  \(k\)-th solution of the eigenvalue problem \eqref{laplacian-pde} in
  \(\Omega\) and \(\lambda_{h,k}\) the k-th eigenvalue of the
  Crouzeix-Raviart discretized problem \eqref{eq:2} in \(\Omega\).
  Then
  \begin{equation}
    \label{eq:3}
    \frac{\lambda_{h,k}}{1 + C_{h}^{2}\lambda_{h,k}} \leq \lambda_{k},
  \end{equation}
  where \(C_{h} \leq 0.1893h\) is a constant.
\end{theorem}
Due to the monotonicity of \eqref{eq:3} this reduces the problem of
lower bounding the fifth eigenvalue of the domain: \(\lambda_{5}\), to
lower bounding the fifth eigenvalue of \eqref{eq:2}:
\(\lambda_{h,5}\).

To get a lower bound for \(\lambda_{h,5}\) we will follow the same
procedure as in
\cite{GomezSerrano-Orriols:negative-hearing-shape-triangle}, making
use of Gershgorin's generalized disks theorem
\cite{Gerschgorin:eigenvalues-theorem}. Before we can do that we need
the following lemma:
\begin{lemma}
  \cite[Lemma
  2.4]{GomezSerrano-Orriols:negative-hearing-shape-triangle}
  \label{lemma:2}
  Let \(v_{1},\dots,v_{m}\) be vectors in \(\mathbb{R}^{m}\) and
  \(s > 0\) such that
  \(|\langle v_{i}, v_{j}\rangle - \delta_{ij}| \leq s\) and suppose
  that \(8ms < 1\). Then there exists an orthonormal set of vectors
  \(w_{1},\dots,w_{m} \in \mathbb{R}^{m}\) such that
  \(\|v_{i} - w_{i}\| \leq \sqrt{3s}\).
\end{lemma}
We will use it in the following way: we first find an approximate
orthonormal basis of eigenvectors \(\{v_{i}\}\) of \(M\), this is
easily done using standard eigenvalue routines. Let \(\tilde{Q}\) be
the matrix that has them as columns, and compute an enclosure for the
almost-diagonal matrix \(\tilde{D} = \tilde{Q}^{T}M\tilde{Q}\). Using
the lemma we can find an orthonormal matrix \(Q\) which is close to
\(\tilde{Q}\). Let \(D = Q^{T}MQ\), which is not necessarily diagonal
but has the same eigenvalues as \(M\) due to the orthogonality of
\(Q\). We can obtain rigorous enclosures of the entries of \(D\) using
\(\tilde{D}\) in the following way
\begin{align*}
  |D_{ij} - \tilde{D}_{ij}| &= |\langle w_{i}, Mw_{j} \rangle - \langle v_{i}, Mv_{j}\rangle|\\
                           &\leq |\langle w_{j} - v_{j}, Mv_{i}\rangle
                             + \langle w_{j} - v_{j}, M(w_{i} - v_{i})\rangle
                             + \langle w_{i} - v_{i}, Mv_{j} \rangle |\\
                           &\leq \sqrt{3s}(\|Mv_{i}\| + \|Mv_{j}\|) + 4s\|M\|_{2},
\end{align*}
with \(s\) is as in the lemma and using the symmetry of \(M\). Observe
that \(\|Mv_{i}\|\) can be computed explicitly and the upper bound
\(\|M\|_{2} \leq \|M\|_{\text{Frob}}\) is easily computed.

Finally applying Gershgorin's generalized disks theorem to the matrix
\(D\), of which we have sharp bounds, we can separate the spectrum of
\(M\) in two components, one of which will contain the first 4
eigenvalues and the other of which will contain the rest. The lower
bound we get for the second component is \(66.2862\) which is thus
also a lower bound for \(\lambda_{h,5}\). A direct application of
Theorem~\ref{thm:1} then gives us the lower bound \(66.0709\) for
\(\lambda_{5}\). Note that this is well above the approximate value
for \(\lambda_{4}\) computed in Section~\ref{sec:finding-candidate}.

\section{Constructing approximate solutions}
\label{sec:constructing-approximation}
In this section we explain how to compute approximations for the first
four eigenfunctions which we will then use to isolate the second
eigenfunction and analyse its nodal line. We will make use of the
Method of Particular Solutions (MPS) to the compute the
approximations. The version of MPS that we give here is due to Betcke
and Trefethen \cite{Betcke-Trefethen:method-particular-solutions,
  Betcke:generalized-svd-mps}. One starts writing the eigenfunction as
a linear combination of functions \(\phi_{i} \, (1 \leq i \leq N)\)
that satisfy the equation \(-\Delta\phi_{i} = \lambda\phi_{i}\) in the
domain but with no boundary conditions, where \(\lambda\) is taken as
a parameter and is part of the problem. The coefficients are then
chosen to approximate the boundary condition we want \(u\) to satisfy.
In the case of a zero Dirichlet boundary condition this amounts to
finding \(\lambda\) and a non-zero linear combination for which the
boundary values are as close to zero as possible. The linear
combination is determined by fixing \(\lambda\) and taking \(m_{b}\)
collocation points on the boundary,
\(\{x_{k}\}_{k = 1}^{m_{b}} \subset \partial \Omega\), then choosing
the linear combination to minimize its values on the collocation
points in the least squares sense. This alone is not quite enough,
increasing the number \(N\) of elements in the basis leads to
existence of linear combinations very close to 0 inside the domain
\(\Omega\). The version by Betcke and Trefethen handles this by also
adding a number \(m_{i}\) of interior points,
\(\{y_{l}\}_{l = 1}^{m_{i}} \subset \Omega\), and taking the linear
combination to stay close to unit norm on these. This is accomplished
by considering the two matrices \(A_{B} = (\phi_{i}(x_{k}))\) and
\(A_{I} = (\phi_{i}(y_{l}))\) which are combined into a matrix whose
\emph{QR} factorization gives an orthonormal basis of these function
evaluations
\begin{equation*}
  A = \begin{bmatrix} A_{B} \\ A_{I} \end{bmatrix} =
  \begin{bmatrix} Q_{B} \\ Q_{I} \end{bmatrix} R =: Q R.
\end{equation*}
The right singular vector \(v\) corresponding to the smallest singular
value \(\sigma = \sigma(\lambda)\) of \(Q_{B}\) for a given
\(\lambda\) is a good candidate for the eigenfunction when
\(\sigma(\lambda)\) is small. In our case we do not require very high
precision for the first, third and fourth eigenvalue and the
approximate values from \eqref{eq:7} gives good enough approximate
eigenfunctions. The second eigenfunction needs slightly higher
precision and this approximate value is not good enough. Instead we
search for a \(\lambda\) around it which minimizes \(\sigma(\lambda)\)
using Brent's method. This minimizer is the value given in Table
\ref{table:finding-approximations}.

We will use three types of basis functions \(\phi_{i}\), all of them
products of Bessel functions and trigonometric functions. They are all
given in polar coordinates centered around a certain point (which may
be different from basis function to basis function). The first type is
the one used in the original version of MPS, it is centered around a
vertex of the domain. If the angle of the vertex is \(\pi/\alpha\)
they take the form
\begin{equation}
  \label{eq:4}
  \phi_{\alpha,k} = J_{\alpha k}(\sqrt{\lambda}r)\sin \alpha k\theta
\end{equation}
in polar coordinates centered around the vertex and \(\theta = 0\)
taken along one of the boundary segments. Unless \(\alpha\) is an
integer these functions have a branch cut in \(\theta\) and are
therefore only suitable for vertices where the branch cut can be
placed outside of the domain, in particular they are not suited for
placement at the vertices of the holes of our domain. The two other
types of functions were recently introduced by Gopal and Trefethen
\cite{Gopal-Trefethen:laplace-solver-detailed}. One is centered around
points (called \textit{charges}) that accumulate root-exponentially
near a vertex of the domain and take the form
\begin{equation}
  \label{eq:5}
  \phi_\text{ext}(r, \theta) = Y_0\left(r \sqrt{\lambda}\right), \qquad
  \phi_\text{ext}^\text{c}(r, \theta) = Y_1\left(r \sqrt{\lambda}\right) \cos \theta, \qquad
  \phi_\text{ext}^\text{s}(r, \theta) = Y_1\left(r \sqrt{\lambda}\right) \sin \theta.
\end{equation}
In this case \(\theta = 0\) is taken along the bisector of the vertex.
The final type is an interior expansion which in our case is placed at
the center of the domain with \(\theta = 0\) taken along the positive
\(x\)-axis
\begin{equation}
  \label{eq:6}
  \phi_0(r, \theta) = J_0\left(r \sqrt{\lambda}\right), \qquad
  \phi_j^\text{c}(r, \theta) = J_j\left(r \sqrt{\lambda}\right) \cos j\theta, \qquad
  \phi_j^\text{s}(r, \theta) = J_j\left(r \sqrt{\lambda}\right) \sin j\theta.
\end{equation}
Here \(J\) and \(Y\) are the Bessel functions of the first and second
kind respectively.

The basis functions will be used in the following way:
\begin{enumerate}
\item At each vertex of the hexagon we place \(2N_{1}\) basis
  functions of the first type \eqref{eq:4}, \(\phi_{\alpha,k}\)
  (\(1 \leq k \leq 2N_{1}\)). The interior angle in this case is
  \(2\pi/3\) so \(\alpha = 3/2\). This gives \(2N_{1}\) free
  coefficients for each vertex for a total of \(12N_{1}\) free
  coefficients.
\item At each vertex of the holes we put \(N_{2}\) charges and basis
  functions of the second kind \eqref{eq:5} around those. This gives
  \(3N_{2}\) free coefficients for each vertex for a total of
  \(54N_{2}\) charges counting all the vertices.
\item In the center of the domain we place basis functions of the
  third type \eqref{eq:6} with \(0 \leq j \leq 6N_{3} - 1\), for a
  total of \(12N_{3} - 1\) free coefficients (\(j = 0\) giving only one
  free coefficient).
\end{enumerate}
In the computations we fix \(n\) and let \(N_{1} = N_{3} = n\) and
\(N_{2} = 3n\). The number of free coefficients is then
\(12n + 54\cdot 3n + 12n - 1 = 186n - 1\). See Figure
\ref{fig:placement} for a schematic of the placement of the basis
functions and charges.

\begin{figure}
  \centering
  \includegraphics[width=0.4\textwidth]{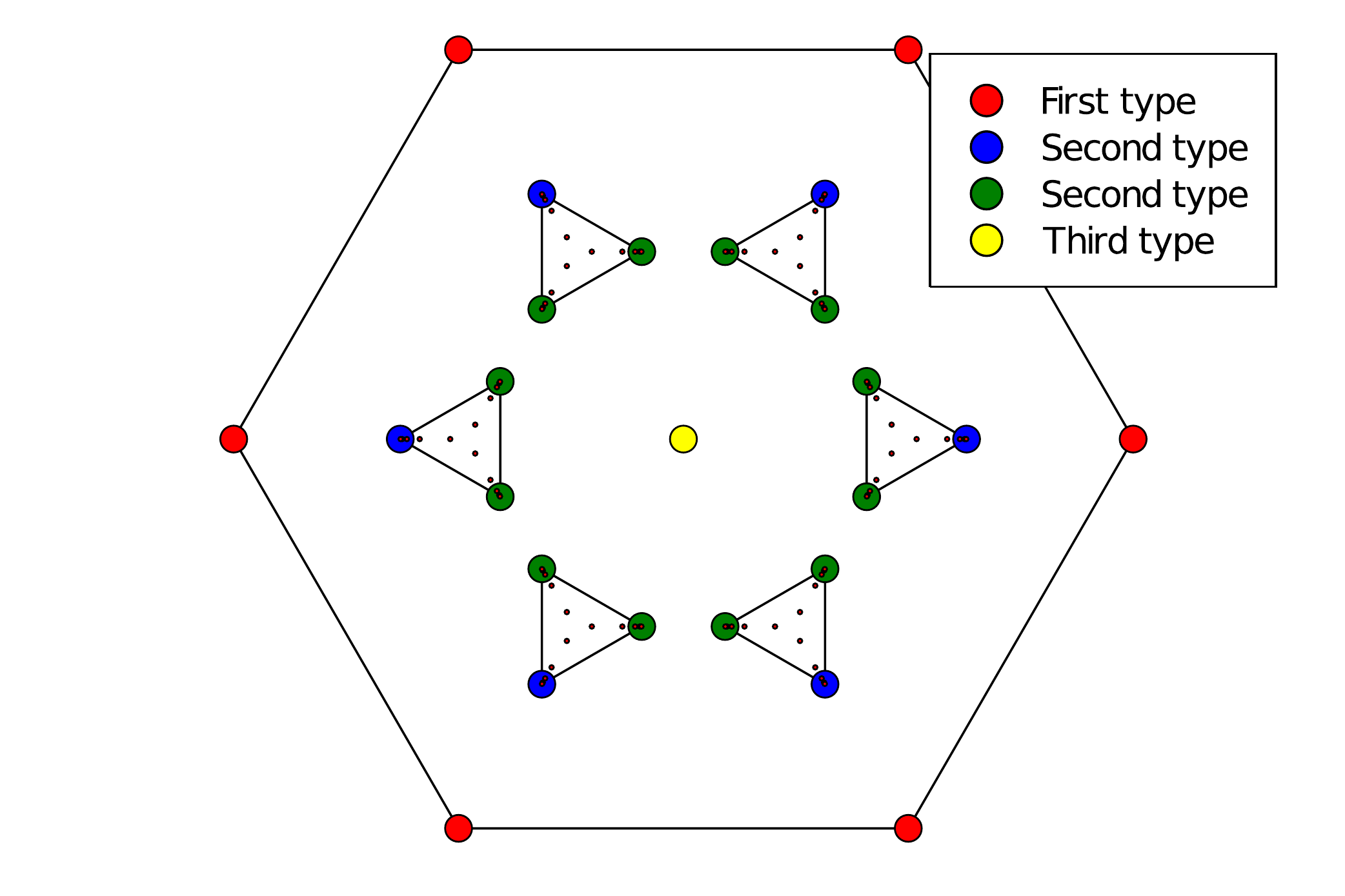}
  \caption{Schematic of placement of charges. The small dots indicate
    placement of charges for the basis functions of the second type,
    only a few charges are shown here. For the first and second
    eigenfunction the expansions with the same color can be taken to
    have the same coefficients due to symmetry. For the third and
    fourth eigenfunction the symmetries are different.}
  \label{fig:placement}
\end{figure}

We can reduce the number of free coefficients substantially by making
use of the symmetries of the domain and the eigenfunctions. We expect
the first and second eigenfunction to have 6-fold symmetry and also
be even with respect to the \(x\)-axis. This allows us to reduce the
number of free coefficients by both fixing some of the expansions to
have the same coefficients and by avoiding the use of some terms in
the expansions which do not satisfy the required symmetry.
\begin{enumerate}
\item For the expansions at the vertices of the hexagon we require
  that they all have the same coefficients, reducing the number of
  free coefficients to \(2N_{1}\) instead of \(12N_{1}\). Furthermore,
  due to their evenness it is enough to consider only even values of
  \(k\), further reducing the number of free coefficients to
  \(N_{1}\).
\item For the expansions at the vertices of the holes the 6-fold
  symmetry allows us to reduce the number of free coefficients to
  \(9N_{2}\). We can reduce it further by using that it has to be
  even. For the expansion at the outer vertex of each hole we only
  have to consider \(\phi_\text{ext}(r, \theta)\) and
  \(\phi_\text{ext}^\text{c}(r, \theta)\), so the number of
  coefficients for that vertex is reduced to \(2N_{2}\). Finally we
  can notice that the expansions at the two inner vertices of the
  holes are symmetric due to the evenness, so must also have the same
  coefficients. This reduces the total number of free coefficients to
  \(5N_{2}\).
\item The inner expansion should be even and have 6-fold symmetry.
  This means we can skip the use of \(\phi_j^\text{s}(r, \theta)\) and
  only have to consider \(j = 6i\) with \(0 \leq i \leq N_{3} - 1\).
  Reducing the number of free coefficients to \(N_{3}\).
\end{enumerate}
This reduces the number of free coefficients from \(184n\) to \(17n\),
more than a 10x improvement.

We expect that the third eigenvalue is even with respect to the
\(x\)-axis and odd with respect to the \(y\)-axis, whereas the fourth
one will satisfy the opposite symmetries. Using these symmetries we
take \(45n\) free coefficients in the first case and \(42n\) in the
second case.

The amount of accuracy we need is different for the different
eigenfunctions. For the first eigenfunction we only need enough
accuracy to separate \(\lambda_{1}\) from \(\lambda_{2}\), and since
they are far away from each other very low accuracy is enough. For the
third and fourth eigenfunction we need slightly higher accuracy since
\(\lambda_{3}\) and \(\lambda_{4}\) are much closer to
\(\lambda_{2}\), we also need slightly higher accuracy to be able to
handle the fact that \(\lambda_{3}\) and \(\lambda_{4}\) presumably
correspond to a double eigenvalue. The second eigenfunction is the
most challenging one: we need much higher accuracy to be able to
isolate the nodal line. Details of the computations are given in Table
\ref{table:finding-approximations}.

\begin{table}
  \centering
  \begin{tabular}{|c|c|c|c|c|}
    \hline
    Eigenfunction & Eigenvalue & \(n\) & Free coefficients & Collocation points \\ \hline
    \(\tilde{u}_{1}\) & \(\tilde{\lambda}_{1} = 31.0432\)& 1 & 17 & 51 \\ \hline
    \(\tilde{u}_{2}\) & \(\tilde{\lambda}_{2} = 63.20833598626884\)& 28 & 476 & 7616 \\ \hline
    \(\tilde{u}_{3}\) & \(\tilde{\lambda}_{3} = 63.7259\)& 6 & 270 & 2160 \\ \hline
    \(\tilde{u}_{4}\) & \(\tilde{\lambda}_{4} = 63.7259\)& 6 & 252 & 2016 \\ \hline
  \end{tabular}
  \caption{Details about computation of the approximate eigenpairs.}
  \label{table:finding-approximations}
\end{table}

We finish this section with four approximate eigenpairs
\((\tilde{\lambda}_{1}, \tilde{u}_{1})\),
\((\tilde{\lambda}_{2}, \tilde{u}_{2})\),
\((\tilde{\lambda}_{3}, \tilde{u}_{3})\) and
\((\tilde{\lambda}_{4}, \tilde{u}_{4})\), as detailed in Table
\ref{table:finding-approximations} and Figure
\ref{fig:approximate-eigenfunctions}. In the next two sections we will
use these approximations to isolate the second eigenvalue and prove
that the nodal line of the second eigenfunction is closed.

\begin{figure}
  \centering
  \includegraphics[width=0.8\textwidth]{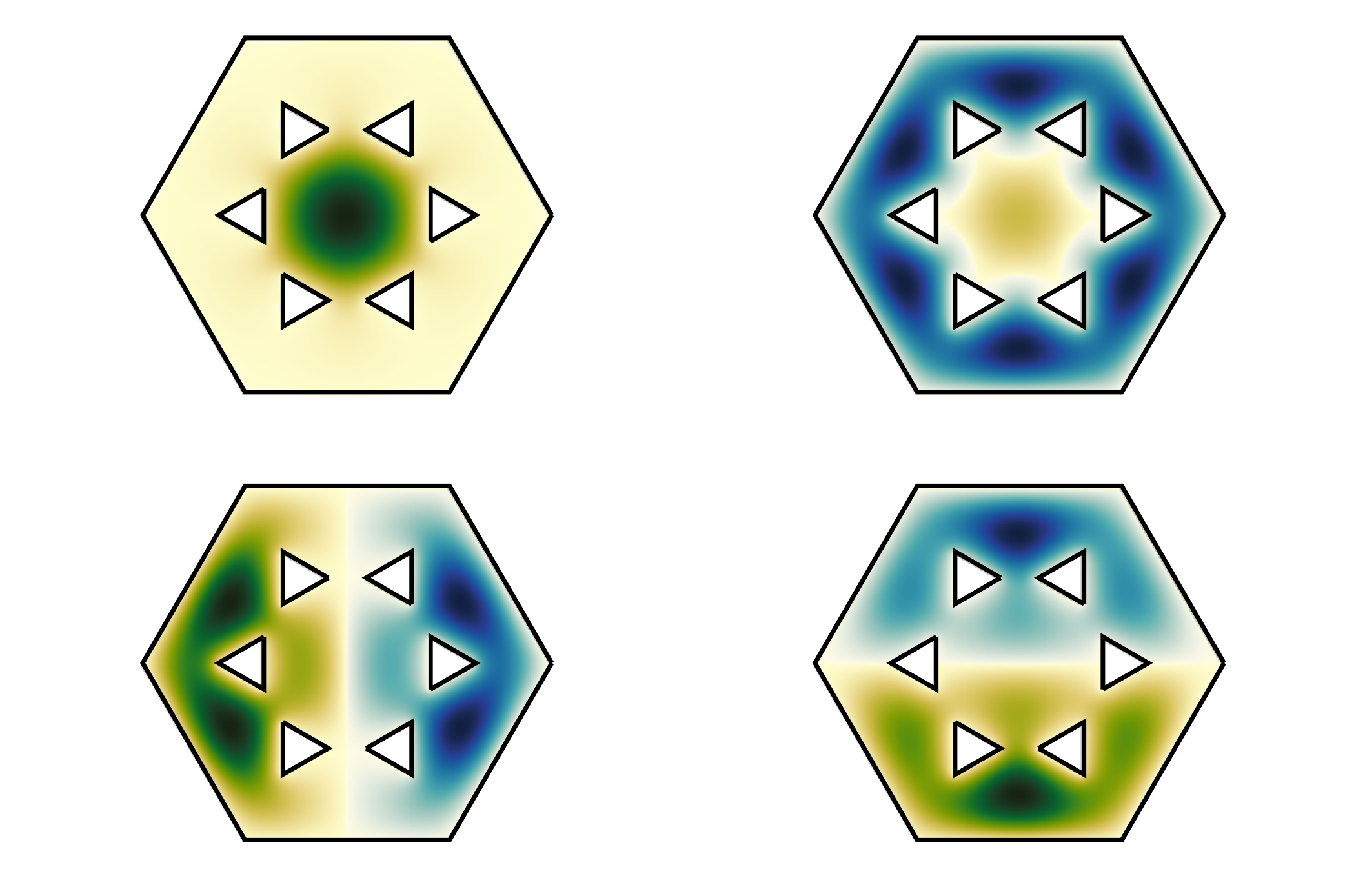}
  \caption{The four approximate eigenfunctions, starting with
    \(\tilde{u}_{1}\) in the top left and ending with
    \(\tilde{u}_{4}\) in the bottom right.}
  \label{fig:approximate-eigenfunctions}
\end{figure}

\section{Isolating the second eigenvalue}
\label{sec:isolating-second}
The main tool in this section is the following theorem by Fox, Henrici
and Moler from the original MPS article
\begin{theorem}
  \label{thm:FoxHenriciMoler}
  \cite{Fox-Henrici-Moler:approximations-bounds-eigenvalues,
    Moler-Payne:bounds-eigenvalues} Let \(\Omega\subset\mathbb{R}^n\)
  be bounded. Let \(\tilde{\lambda}\) and \(\tilde{u}\) be an
  approximate eigenvalue and eigenfunction---that is, they satisfy
  \(\Delta \tilde{u}+\tilde{\lambda} \tilde{u}=0\) in \(\Omega\) but
  not necessarily \(\tilde{u} = 0\) on~\(\partial\Omega\). Define
  \begin{equation}\label{mu:bound}
    \mu = \frac{\sqrt{|\Omega|}\sup_{x \in \partial \Omega}|\tilde{u}(x)|}{\|\tilde{u}\|_2}.
  \end{equation}
  where \(|\Omega|\) is the area of the domain. Then there exists an
  eigenvalue \(\lambda\) such that
  \begin{equation}
    \label{eq:eigenvalue-bound}
    \frac{|\tilde{\lambda} - \lambda|}{\lambda} \leq \mu.
  \end{equation}
\end{theorem}
This allows us to calculate enclosures of the approximate eigenvalues
computed in the previous section. To do that we need to find an
upper bound for \(\mu\) in \eqref{mu:bound}, which in turn means we need an
upper bound for the maximum on the boundary and a lower bound for the
\(L^{2}\) norm of the approximate eigenfunctions.

The upper bound for the maximum is computed using a combination of
interval arithmetic and Taylor expansions. The sides are divided into
small segments and on the midpoint of each segment a Taylor polynomial
of the approximate eigenfunction is computed. The maximum of the
polynomial is then bounded using classical interval arithmetic and the
error term for the polynomial is added. For more details see
\cite[Section 2.1]{Dahne-Salvy:enclosures-eigenvalues}. Thanks to the
symmetries satisfied by the approximate eigenfunctions we only have to
check the maximum on parts of the boundary.

For lower bounding the norm we use the same procedure as in
\cite{GomezSerrano-Orriols:negative-hearing-shape-triangle} and
\cite{Dahne-Salvy:enclosures-eigenvalues}. Consider a subset of the
domain, \(\Omega' \subset \Omega\). If \(u\) does not vanish in
\(\Omega'\) then without loss of generality it can be assumed to be
positive there and then, since \(-\Delta u = \lambda u > 0\), \(u\) is
superharmonic in \(\Omega'\) and satisfies
\(\inf_{\Omega'} u \geq \inf_{\partial\Omega'} u\). Thus a lower bound
for \(|u|\) on \(\partial\Omega'\) yields a lower bound for \(|u|\)
inside \(\Omega'\). To determine that \(u\) does not vanish on
\(\Omega'\) a lower bound for \(|u|\) on \(\partial\Omega'\) is
computed with the same techniques as when upper bounding the maximum.
Once it is determined
that \(u\) has a fixed sign on \(\partial\Omega'\) (which we assume to
be positive) the key observation is that
\(u\) cannot be negative inside \(\Omega'\) if \(\Omega'\) is small
enough. Indeed, if \(\Omega'' \subset \Omega'\) is a maximal domain
where \(u < 0\), then \(u = 0\) on \(\partial\Omega''\) and thus
\(\lambda\) is an eigenvalue for \(\Omega''\). If the area of
\(\Omega'\) is small enough this scenario can be ruled out using the Faber-Krahn
inequality. 
Details about the
choice of \(\Omega'\) for the different eigenfunctions are given in
Section \ref{sec:details-of-implementation}.

Upper bounds for \(\mu\) together with upper bounds for the maximum on
the boundary and enclosures of the eigenvalues are given in Table
\ref{table:mu-bound}. From Section \ref{sec:separating-first-four} we
know that there are at most four eigenvalues below \(66.0709\). If all
of the four enclosures were isolated we could have concluded that we
had isolated all four eigenvalues and we would know their indices in
the spectrum. However the enclosures coming from \(\tilde{u}_{3}\) and
\(\tilde{u}_{4}\) overlap and we have to handle that.

\begin{table}
  \centering
  \begin{tabular}{|c|c|c|c|}
    \hline
    Eigenfunction & \(\mu \leq \) & \(\sup_{x \in \partial\Omega} \leq\) & Enclosure for eigenvalue \\ \hline
    \(\tilde{u}_{1}\) & 0.14 & 0.012 & \([30 \pm 6.1]\) \\ \hline
    \(\tilde{u}_{2}\) & \(8.26\cdot 10^{-5}\) & \(1.01\cdot 10^{-6}\) & \([63.21 \pm 6.89\cdot 10^{-3}]\) \\ \hline
    \(\tilde{u}_{3}\) & 0.00186 & \(2.68\cdot 10^{-5}\) & \([64 \pm 0.393]\) \\ \hline
    \(\tilde{u}_{4}\) & 0.00215 & \(3.1\cdot 10^{-5}\) & \([64 \pm 0.411]\) \\ \hline
  \end{tabular}
  \caption{Upper bounds of \(\mu\) and the value on the boundary of
    the four approximate eigenfunctions together with computed
    enclosures for the corresponding eigenvalue.}
  \label{table:mu-bound}
\end{table}

\subsection{Handling the 2-cluster}
\label{sec:handling-double-eigenvalue}
Since the enclosures coming from \(\tilde{\lambda}_{3}\) and
\(\tilde{\lambda}_{4}\) overlap we cannot be sure that they indeed
correspond to two different eigenvalues. From the plots of the
corresponding eigenfunctions in Figure
\ref{fig:approximate-eigenfunctions} it does indeed seem like they do
correspond to different eigenvalues but what we will prove is a
slightly weaker statement which is enough for what we want to do. We
will prove that there are at least two eigenvalues in an interval
slightly larger than the two enclosures.

The proof will be based on the fact that \(\tilde{u}_{3}\) and
\(\tilde{u}_{4}\) are not proportional to each other so as long as the
error bounds for them are sufficiently small this would imply that the corresponding exact solutions can not possibly correspond to the same eigenfunction. The error bounds we will use are given by the following
theorem:
\begin{theorem}
  \cite[Theorem 3]{Moler-Payne:bounds-eigenvalues}
  \label{thm:L-inifinity-bounds}
  Let \(\Omega\subset\mathbb{R}^2\) be bounded. Let
  \(\tilde{\lambda}\) and \(\tilde{u}\) be an approximate eigenvalue
  and eigenfunction---that is, they satisfy
  \(\Delta \tilde{u}+\tilde{\lambda} \tilde{u}=0\) in \(\Omega\) but
  not necessarily \(\tilde{u} = 0\) on~\(\partial\Omega\). Let
  \begin{equation*}
    \mu = \frac{\sqrt{|\Omega|}\sup_{x \in \partial \Omega}|\tilde{u}(x)|}{\|\tilde{u}\|_2}.
  \end{equation*}
  Then there is an eigenvalue \(\lambda_{k}\) satisfying the same
  bounds as in Theorem \ref{thm:FoxHenriciMoler} and a corresponding
  eigenfunction \(u_{k}\). Let
  \begin{equation*}
    g(x) = \left(\int_{\Omega}G(x, y)^{2}dy\right)^{1/2}
  \end{equation*}
  where \(G(x, y)\) is the Green's function and
  \begin{equation*}
    \alpha = \min_{\lambda_{n} \not= \lambda_{k}} \frac{|\lambda_{n} - \tilde{\lambda}|}{|\lambda_{n}|}.
  \end{equation*}
  Then for any \(x \in \Omega\) we have the following bound for
  \(\tilde{u}\)
  \begin{equation*}
    |\tilde{u}(x) - u_{k}(x)| \leq
    \left(\sup_{x \in \partial \Omega}|\tilde{u}(x)|\right)\left(1 + g(x)\tilde{\lambda}\left(
      \frac{1}{1 - \mu} + \frac{1}{\alpha}\left(1 + \frac{\mu^{2}}{\alpha^{2}}\right)
    \right)\right).
  \end{equation*}
\end{theorem}
To compute this bound we need an upper bound for \(g(x)\) and a lower
bound for \(\alpha\). Using \cite[Theorem 2.4, Corollary
2.2]{Bandle:isoperimetric-inequalities-book} with $p=2$, we obtain
that
\begin{align}\label{gbound}
g(x) \leq \frac{1}{4\pi}\sqrt{2|\Omega|}
\end{align}
\begin{remark}
  The bound \eqref{gbound} is far from optimal (especially when
  $x \to \partial \Omega$), but we have preferred to keep a simple,
  uniform bound.
\end{remark}

\begin{remark}
It may be possible to obtain $L^\infty$ bounds by deriving higher order estimates and using the Sobolev embedding as in \cite{Plum:H2-estimates-elliptic-bvp} but our method does not require any additional explicit constant.
\end{remark}

We are now ready to handle the 2-cluster. Let \(\Lambda\) be the union
of the enclosures of the eigenvalues \(\tilde{\lambda}_{3}\) and
\(\tilde{\lambda}_{4}\) as given by Theorem \ref{thm:FoxHenriciMoler}.
Then we have the following result.

\begin{lemma}
  \label{lemma:2-cluster}
  Let \(r\) be the radius of \(\Lambda\) and let \(\Lambda'\) be the
  interval with the same midpoint as \(\Lambda\) and radius
  \(\frac{17r}{16}\). Then there are at least two eigenvalues in the
  interval \(\Lambda'\).
\end{lemma}
\begin{proof}
  The proof is by contradiction. Assume that there is only one
  eigenvalue in \(\Lambda'\). From Theorem
  \ref{thm:L-inifinity-bounds} we get that there are two
  eigenfunctions \(u_{3}\) and \(u_{4}\) with eigenvalues
  \(\lambda_{3}\) and \(\lambda_{4}\) in \(\Lambda\). By our assumption, we must have
  \(\lambda_{3} = \lambda_{4}\) and also \(u_{3} = Cu_{4}\) for some
  \(C \in \mathbb{R}\).

  Since \(\lambda_{3} = \lambda_{4}\) lie in \(\Lambda\) and we assume
  that there are no other eigenvalues in \(\Lambda'\) we get that
  \(\alpha\) in Theorem \ref{thm:L-inifinity-bounds} is lower bounded
  by \(\frac{r}{16}\). With a lower bound for \(\alpha\) and upper
  bounds for \(g(x)\), \(\sup_{x \in \partial\Omega} |\tilde{u}(x)|\)
  and \(\mu\) we can use Theorem \ref{thm:L-inifinity-bounds} to
  compute an upper bound \(d_{3}\) for
  \(|\tilde{u}_{3}(x) - u_{3}(x)|\) and \(d_{4}\) for
  \(|\tilde{u}_{4}(x) - u_{4}(x)|\).

  Now consider the points \(p_{1} = (1/2, 1/2)\) and
  \(p_{2} = (-1/2, 1/2)\). We can evaluate \(\tilde{u}_{3}\) and
  \(\tilde{u}_{4}\) at these points. We find that
  \(\tilde{u}_{3}(p_{1}) + d_{3} < 0 \text{ and } \tilde{u}_{4}(p_{1})
  + d_{4} < 0\) so both \(u_{3}\) and \(u_{4}\) must have the same
  sign at \(p_1\), this means that \(u_{3} = Cu_{4}\) for some
  \(C > 0\) and in particular \(u_{3}\) and \(u_{4}\) must have the
  same sign everywhere. Furthermore we find that
  \(\tilde{u}_{3}(p_{2}) - d_{3} > 0\) so \(u_{3}\) is positive at
  \(p_2\) but \(\tilde{u}_{4}(p_{2}) + d_{3} < 0\) so \(u_{4}\) must
  be negative at \(p_2\). This would contradict that \(u_{3}\) and
  \(u_{4}\) have the same sign and hence there must be at least two
  eigenvalues in \(\Lambda'\).
\end{proof}

\begin{remark}
  While Lemma \ref{lemma:2-cluster} proves that there are at least two
  eigenvalues in \(\Lambda'\) it does not prove that there has to be a
  double eigenvalue.
\end{remark}

\begin{remark}
  We remark that it is fundamental for the cluster to be a 2-cluster
  in the proof. Otherwise, there are no lower bounds on \(\alpha\)
  available, in order to apply Theorem \ref{thm:L-inifinity-bounds}.
  We outline a different (though more costly) strategy for the general
  case.

  The idea is to make use of the symmetries of the domain and the
  approximate eigenfunctions and compute bounds on the smaller domain
  \(\Omega' \subset \Omega\) which are then extended to \(\Omega\).
  For example, for \(\tilde{u}_{3}\) we could do the following.
  Consider \(\tilde{u}_{3}\) as an approximate eigenfunction on the
  domain \(\Omega'\) given by the right half of \(\Omega\). If we can
  compute \(L^{\infty}\) bounds for \(\tilde{u}_{3}\) on \(\Omega'\)
  then these bounds also apply to \(\Omega\) extending by symmetry.
  When computing bounds on \(\Omega'\) using Theorem
  \ref{thm:L-inifinity-bounds} we would get exactly the same bound for
  \(\mu\) since by construction, \(\tilde{u}_{3}\) is identically
  equal to zero on the only additional boundary and the factor
  difference in norm and area scale out. In \(\Omega'\),
  \(\tilde{u}_{3}\) corresponds to a simple eigenvalue and we could
  get a lower bound for \(\alpha\) without having to deal with a
  cluster of eigenvalues. This would require more work than Lemma
  \ref{lemma:2-cluster} since we would have to control the spectrum of
  \(\Omega'\) as well, but it has the benefit that it would work for
  clusters with more than two eigenvalues.
\end{remark}

We have now proved that there must be one eigenvalue in
\([30 \pm 6.1]\), one in \([63.21 \pm 6.89\cdot 10^{-3}]\) and two in
\(\Lambda'\). Since we have already proved that there are at most four
eigenvalues below \(66.0709\) we can conclude that we have found them
all. In particular this means that
\((\tilde{\lambda}_{2}, \tilde{u}_{2})\) must indeed correspond to the
second eigenpair and the eigenvalue closest to it lies in
\(\Lambda'\).

\section{Analysing the nodal line and conclusion of the proof}
\label{sec:nodal-line}
In this section we will assume that the sign of \(\tilde{u}_{2}\) is
taken such that it is positive at the center. To prove that the nodal
line is isolated we will construct a closed path \(\Gamma\) around the
center of the domain. Let \(\tilde{\Omega}\) be the region enclosed by
\(\Gamma\), \(\Gamma\) will be taken such that \(\tilde{\Omega}\) does
not intersect the boundary of \(\Omega\). We will prove that \(u_{2}\)
is strictly negative on \(\Gamma\) and positive in some point in
$\tilde{\Omega}$. This would imply that \(u_{2}\) changes sign
somewhere inside of $\tilde{\Omega}$ and hence at least part of the
nodal line must be inside it. Since \(u_{2}\) is strictly negative on
\(\Gamma\) we also get that the nodal line can not cross \(\Gamma\),
thus there is a closed curve in $\tilde{\Omega}$ belonging to the
nodal line. By Courant's nodal domain Theorem
\cite{Courant-Hilbert:mathematical-physics-vol1-book} there can only
be two connected components and hence the whole nodal line is closed
and contained in $\tilde{\Omega}$, therefore not touching $\partial \Omega$.

The choice of \(\Gamma\) is seen in Figure \ref{fig:gamma}. It
consists of a straight line between \((d, 0)\) and
\((d, \tan(\pi/6)d)\) (the solid part of the red line in Figure
\ref{fig:gamma}) which is mirrored in the \(x\)-axis and extended in a
6-fold symmetry way. Since \(\tilde{u}_{2}\) satisfies the same
symmetries it is enough to bound it on this straight line to get a
bound on \(\Gamma\). The value of \(d\) is chosen so that the value of
\(\tilde{u}_{2}(d, 0)\) is as negative as possible. We picked
\(d = 0.38483177115481165\).
\begin{figure}
  \centering
  \includegraphics[height=4cm]{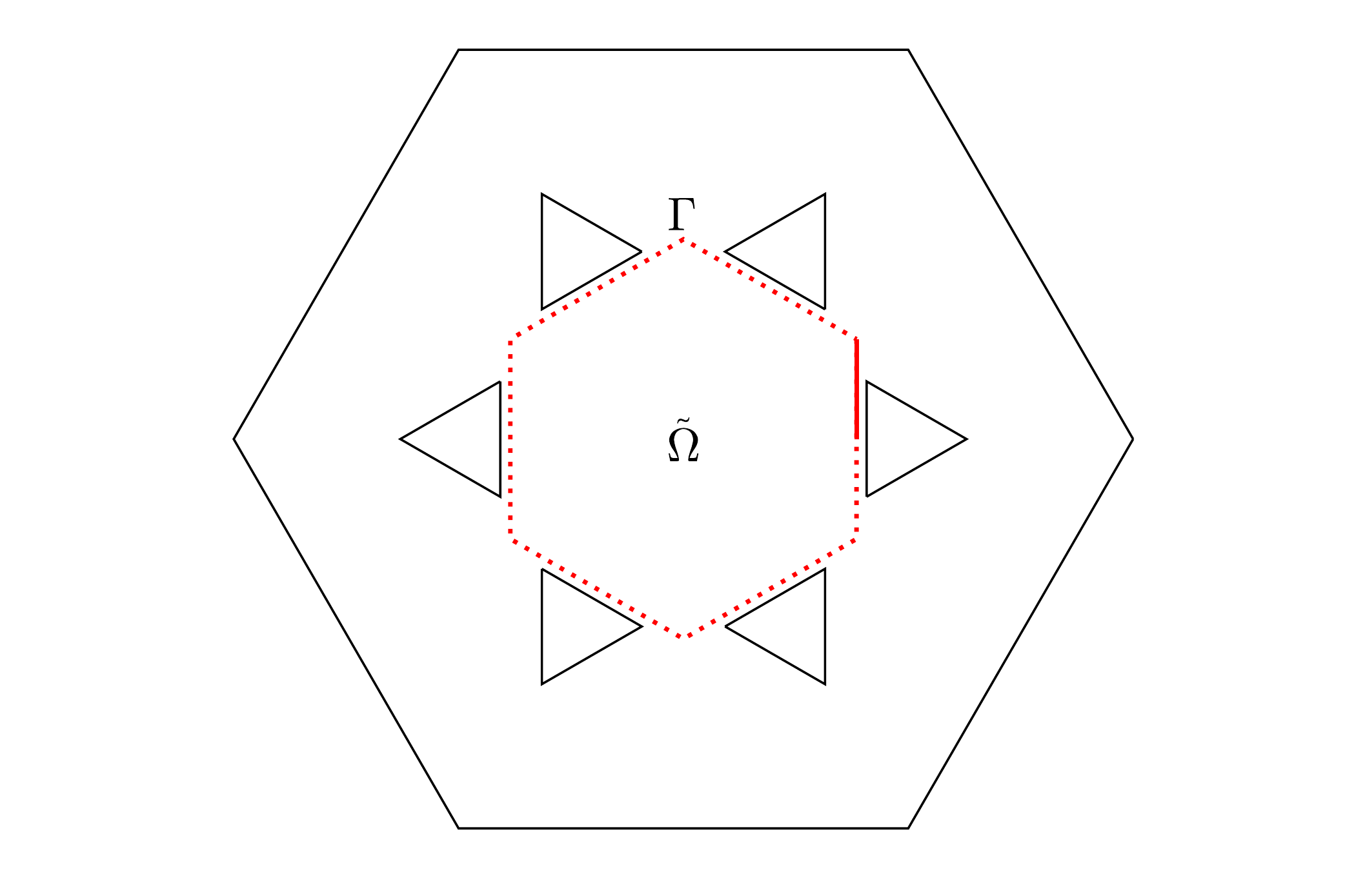}
  \caption{Shows the path \(\Gamma\), enclosing the region
    \(\tilde{\Omega}\), in red. By symmetry arguments it is enough to
    consider the value of \(\tilde{u}_{2}\) on the solid part of the
    line.}
  \label{fig:gamma}
\end{figure}

We compute an upper bound for \(\tilde{u}_{2}\) on \(\Gamma\) using
the same techniques as when computing the bound on the boundary. In
this case it is much less costly since there are less cancellations
between the terms of \(\tilde{u}_{2}\) and we are further away from
the charge points. We get the upper bound
\(\gamma = -4.4929\cdot 10^{-5}\). We are left to prove that the error
bound for \(\tilde{u}_{2}\) from Theorem \ref{thm:L-inifinity-bounds}
is less than \(|\gamma|\) to conclude that \(u_{2}\) must be negative
on all of \(\Gamma\).

From Lemma \ref{lemma:2-cluster} we know that the closest eigenvalue
to \(\tilde{\lambda}_{2}\) is in the interval \(\Lambda'\). This gives
us the bound \(\alpha \geq 0.3727\). Together with the upper bound for
\(g(x)\) from \eqref{gbound} and the upper bounds for \(\mu\) and
\(\max_{x \in \partial\Omega}|\tilde{u}_{2}(x)|\) from Table
\ref{table:mu-bound}, Theorem \ref{thm:L-inifinity-bounds} gives us
\begin{equation*}
  |u_{2}(x) - \tilde{u}_{2}(x)| \leq 4.2162 \cdot 10^{-5}.
\end{equation*}
Since \(|\gamma| > 4.2162 \cdot 10^{-5}\) we conclude that \(u_{2}\)
is negative on all of \(\Gamma\).

Finally we have to prove that \(u_{2}\) is positive at some point in $\tilde{\Omega}$, for that we just take the point \((1/10, 0)\) for
which we have
\(\tilde{u}_{2}(1/10, 0) \in [0.01342 \pm 3.27\cdot 10^{-6}]\). Since
this is greater than the error bound for \(\tilde{u}_{2}\) we can
conclude that \(u_{2}\) is positive on at least one point in $\tilde{\Omega}$, which implies that the nodal line for \(u_{2}\) must be fully contained
in $\tilde{\Omega}$ and hence is closed. This finishes the proof of Theorem \ref{MainThm}.

\section{Details of the Implementation}
\label{sec:details-of-implementation}

The code\footnote{Available at
  \url{https://github.com/Joel-Dahne/PaynePolygon.jl}.} for the
computer assisted parts is implemented in Julia
\cite{Bezanson-Edelman-Karpinski-Shah:julia} and makes use of Arb
\cite{Johansson:Arb} for the rigorous parts of the numerics. The part
related to Section \ref{sec:separating-first-four} is implemented from
scratch whereas the rest relies heavily on the code from
\cite{Dahne-Salvy:enclosures-eigenvalues} with added support for new
geometries and new types of basis functions. Roughly the code is
divided into three parts:
\begin{enumerate}
\item The code for computing the lower bound of \(\lambda_{h,5}\)
  using the FEM method. The computation of \(Q\) uses standard
  eigenvalue routines and the verification is then easily done using
  Arb. It is complicated by the fact that the matrix \(M\) is very
  large (\(6084 \times 6084\)) and that standard double precision
  computations give an \(s\) in Lemma \ref{lemma:2} which is not quite
  small enough for the separation to work. To get higher precision we
  do the computations with so called double-doubles
  \cite{Shewchuk-adaptive-precision-fpa} through David K. Zhang's
  Julia library for multi-float
  computations\footnote{https://github.com/dzhang314/MultiFloats.jl}.
  This has the drawback that there are less specialised eigenvalue
  routines and the computations thus take longer, hours instead of
  minutes, but otherwise it works well.
\item The code for computing the approximate eigenfunctions is mostly
  the same as in \cite{Dahne-Salvy:enclosures-eigenvalues} with some
  modifications for the planar case. The collocation points are placed
  root-exponentially close to the vertices, in line with the
  recommendations of \cite{Gopal-Trefethen:laplace-solver-detailed}.
  For the lightning charges we use standard double precision. Between
  the expansions at the vertices of the hexagon and the expansion at
  the center we see very large cancellations and we found that
  increasing the precision helped. For the first eigenfunction 64 bits
  was enough and for the third and fourth we found 128 bits to
  suffice. For the third eigenfunction the cancellations were even
  larger and we ended up using 384 bits of precision.
\item The computations required for the verification of the
  approximate solutions, mostly the lower bound of the norm and the
  upper bound on the maximum for the approximate eigenfunctions, is
  for the most part exactly the same as in
  \cite{Dahne-Salvy:enclosures-eigenvalues}. The subset of the domain
  where the norm is lower bounded is tuned for each eigenfunction and
  given in Figure \ref{fig:norm-subsets}.
\end{enumerate}

The calculations were run on relatively old hardware, the FEM part on
an Intel Xeon CPU E5-2620 with 48GB of memory and the rest on an Intel
Core i7-3770 with 16GB of memory. The most time consuming part of the
computations is the FEM method. The computation of \(\tilde{Q}\) took
around 30 hours and then an additional 2 hours for the verification
with Arb. Computation of the approximate eigenfunctions took around 10
seconds for \(\tilde{u}_{1}\), 6 hours for \(\tilde{u}_{2}\) and 5
minutes each for \(\tilde{u}_{3}\) and \(\tilde{u}_{4}\). The bounds
for the norm took less than 20 minutes in total whereas the bound for
the maximum took around 10 seconds for \(\tilde{u}_{1}\), 6 hours for
\(\tilde{u}_{2}\) and half an hour each for \(\tilde{u}_{3}\) and
\(\tilde{u}_{4}\).

\begin{figure}
  \centering
  \includegraphics[width=0.8\textwidth]{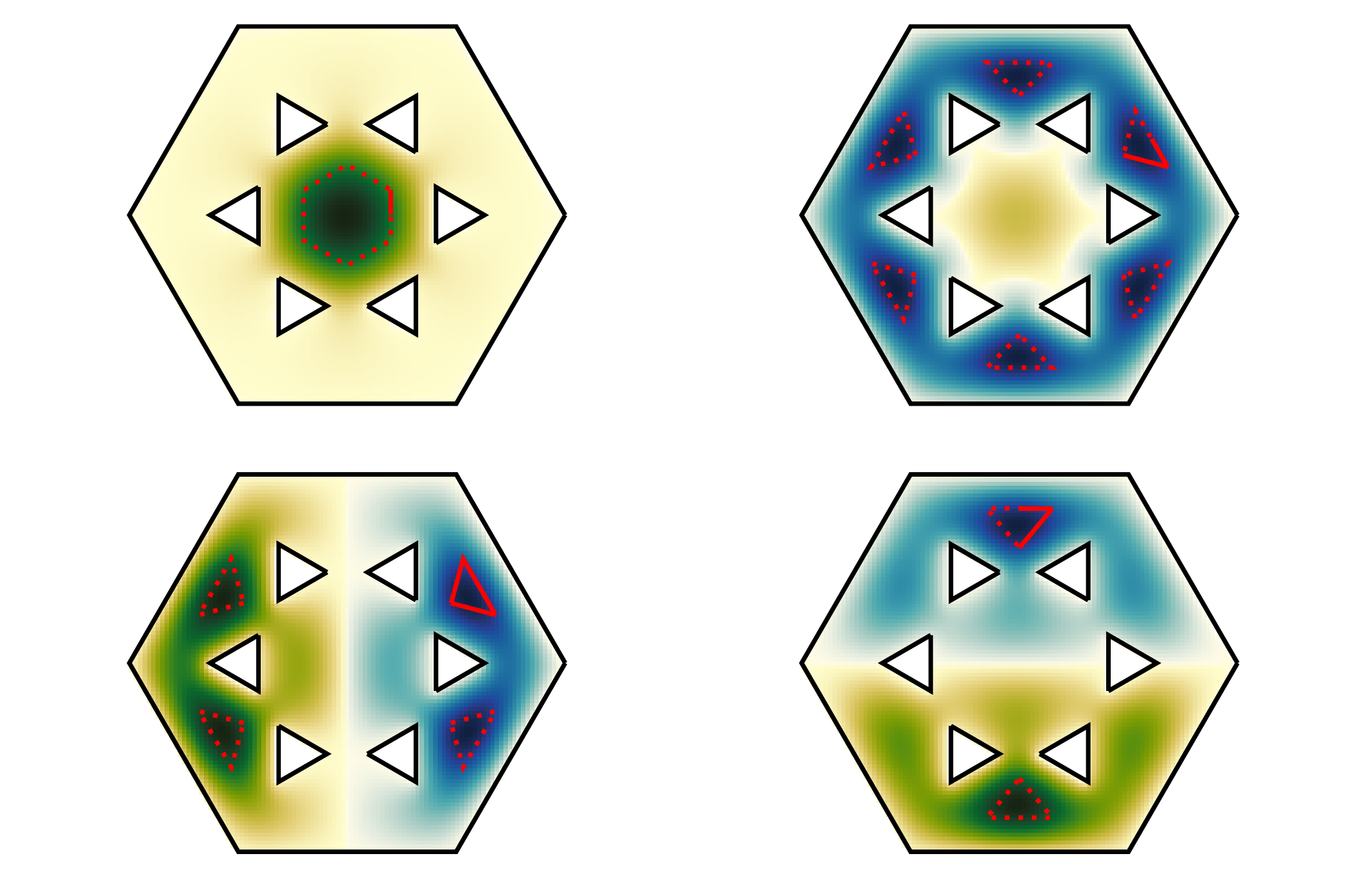}
  \caption{The part of the domain where the norm of each eigenfunction
    is lower bounded is marked with red lines. For symmetry reasons
    it is enough to bound the eigenfunctions on the parts with solid
    lines.}
  \label{fig:norm-subsets}
\end{figure}

\section*{Acknowledgments}

JD was partially supported by the European Research Council through ERC-StG-852741-CAPA. JGS was partially supported by the European Research Council through ERC-StG-852741-CAPA, by NSF through Grant NSF DMS-1763356 and by the Princeton Summer Program for Mathematics Majors. KH was partially supported by the Princeton Summer Program for Mathematics Majors. We thank Uppsala University for computing facilities (Haddock Cluster).

\bibliographystyle{acm}
\bibliography{references}

\begin{tabular}{l}
\textbf{Joel Dahne} \\
{Department of Mathematics} \\
{Uppsala University} \\
{L\"agerhyddsv\"agen 1, 752 37, Uppsala, Sweden} \\
{Email: joel.dahne@math.uu.se} \\ \\
\textbf{Javier G\'omez-Serrano}\\
{Department of Mathematics} \\
{Brown University} \\
{Kassar House, 151 Thayer St.} \\
{Providence, RI 02912, USA} \\ \\
{and} \\ \\
{Departament de Matem\`atiques i Inform\`atica} \\
{Universitat de Barcelona} \\
{Gran Via de les Corts Catalanes, 585} \\
{08007, Barcelona, Spain} \\
{Email: javier\_gomez\_serrano@brown.edu, jgomezserrano@ub.edu} \\ \\
\textbf{Kimberly Hou}\\
{Department of Mathematics}\\
{Princeton University} \\
{Fine Hall, Washington Rd,}\\
{Princeton, NJ 08544, USA}\\
{Email: klhou@princeton.edu}\\
\end{tabular}

\end{document}